\documentclass[12pt,reqno]{amsart}
\usepackage[top=2cm,bottom=2cm,right=2.5cm,left=2.5cm]{geometry}
\usepackage{amssymb}
\usepackage{amsmath, amsthm, amscd, amsfonts, amssymb, graphicx, color}
\usepackage[bookmarksnumbered, colorlinks, plainpages]{hyperref}

\usepackage{hyperref}

\newtheorem{theorem}{Theorem}[section]
\newtheorem{lemma}[theorem]{Lemma}
\newtheorem{remark}[theorem]{Remark}
\newtheorem{proposition}[theorem]{Proposition}
\newtheorem{corollary}[theorem]{Corollary}

\theoremstyle{definition}
\newtheorem{definition}[theorem]{Definition}
\newtheorem{example}[theorem]{Example}

\begin{document}
	
	\setcounter{page}{1}
	
	\title[$K-$frames for Super Hilbert Spaces]{$K-$frames for Super Hilbert Spaces}

	\author[N. Khachiaa]{Najib Khachiaa$^*$}
	
	\address{Department of Mathematics Faculty of Sciences, University of Ibn Tofail, B.P. 133, Kenitra, Morocco}
	\email{\textcolor[rgb]{0.00,0.00,0.84}{khachiaa.najib@uit.ac.ma}}
	\date{%01/10/2019; %Accepted: zzzzzz.
		\newline \indent $^{*}$ Corresponding author}
	
	
	\begin{abstract}
		Let $H_1$ and $H_2$ be two Hilbert spaces, $K$ and $L$ be bounded operatrors on $H_1$ and $H_2$ respectively. In this paper we study the relationship between $K$-frames for $H_1$ and $L$-frames for $H_2$ and $K\oplus L$-frames for $H_1\oplus H_2$. The $K\oplus L$-minimal frames and  $K\oplus L$-orthonormal bases  for $H_1\oplus H_2$ are also studied.
	\end{abstract}
	\maketitle
	
	\baselineskip=12.4pt

	\section{Introduction and preliminaries}
Frames in Hilbert spaces were introduced by Duffin and Schaeffer in 1952 \cite{6}, in the context of non-harmonic Fourier series. After a few years, in 1986, the frames were brought back to life by Daubechies, Grossman and Meyer \cite{5}.
Frames have interesting properties that make them useful tools in signal processing, image processing, coding theory, sampling theory and much more. Recently, the theory of frames has undergone several generalizations: g-frames, k-frames. In this work, we will study the theory of K-frames in super Hilbert spaces.\\  \\

In this paper $H_1$ and $H_2$ are separable Hilbert spaces. $H_1\oplus H_2$ is the direct sum  of Hilbert spaces $H_1$ and $H_2$. $H_1\oplus H_2$ equipped with the inner product $<x\oplus y, a\oplus b>\\:=<x,a>_{H_1}+<y,b>_{H_2}$, for all $x,a\in H_1$ and $y,b\in H_2$, is clearly a Hilbert space which we call the super Hilbert space of $H_1$ and $H_2$. $B(H_1,H_2)$ denote the collection of all bounded linear operators from $H_1$ in $H_2$. For a Hilbert space $H$, $B(H)$ denote the algebra of all bounded linear operators on $H$. For $K\in B(H_1)$ and $L\in B(H_2)$, we denote by $K\oplus L$ the bounded linear operator on $H_1\oplus H_2$ defined  for all $x\in H_1$ and $y\in H_2$ by: $K\oplus L(x\oplus y)=K(x)\oplus L(y)$. For an operator $T$, $R(T)$ and $N(T)$ denote the range and the kernel, respectively, of $T$.\\ 
In what follows, without any possible confusion,  all inner products are denoted by the same notation $\langle.,.\rangle$ and also all norms are denoted by the same notation $\parallel.\parallel$.\\ \\
The following is one of the most useful theorems in $K$-frame theory which known by Douglas theorem:
\begin{theorem}\cite{4}
Let $K\in B(H_1,H)$, $L\in B(H_2,H)$, the following statements are equivalent:
\begin{enumerate}
\item[i.] $R(K)\subset R(L)$.
\item[ii.] There exists a constant $c>0$ such that $KK^*\leq cLL^*$.
\item[iii.] $K=LX$ for some $X\in B(H_1,H_2)$.\\
\end{enumerate}
\end{theorem}
\begin{definition}\cite{3}
A sequence $\{x_n\}_{n\geqslant}$ in $H$ is called to be Bessel sequence in  $H$ if there exist a constant $B>0$ such that for all $x\in H$,  $$\displaystyle{\sum_{n=1}^{+\infty}\vert \langle x,x_n\rangle\vert^2}\leq B\|x\|^2.$$
\end{definition}
The following operators are well defined and bounded whenever the sequence is a Bessel one.
\begin{definition}\cite{3}
Let $\{x_n\}_{n\geqslant 1}$ be a Bessel sequence.
\begin{enumerate}
\item[i.] The synthesis operator (or the pre-frame operator) of $\{x_n\}_{n\geqslant 1}$, denoted by $T$, is the bounded operator defined by $$T:\begin{array}{rcl}
\ell^2&\longrightarrow& H\\
\{a_n\}_{n\geqslant 1}&\mapsto& \displaystyle{\sum_{n=1}^{+\infty}a_n x_n}.
\end{array}$$
\item[ii.] The analysis operator ( or the frame transforme) of $\{x_n\}_{n\geqslant 1}$ is the adjoint of its synthesis operator. It is  often denoted by $\theta$. It is  defined explicitly by:  $$\theta:\begin{array}{rcl}
H&\longrightarrow& \ell^2\\
x&\mapsto& \{\langle x,x_n\rangle\}_{n\geqslant 1}.
\end{array}$$
\item[iii.] The frame operator, denoted by $S$, is the the composite of $T$ and $\theta$. It is defined  explicitly by:
 $$S:\begin{array}{rcl}
H&\longrightarrow& H\\
x&\mapsto& \displaystyle{\sum_{n=1}^{+\infty}\langle x, x_n\rangle x_n}.
\end{array}$$\\
\end{enumerate}
\end{definition}
\begin{definition}\cite{3}
A sequence $\{x_n\}_{n\geqslant}$ in $H$ is said to be Frame for $H$ if there exist two constants $A,B>0$ such that for all $x\in H$,  $$A\|x\|^2\leq \displaystyle{\sum_{n=1}^{+\infty}\vert \langle x,x_n\rangle \vert^2}\leq B\|x\|^2.$$
If $A=B$, we say that $\{x_n\}_{n\geqslant 1}$ is a tight frame, if $A=B=1$ the frame is called Parseval frame.\\
\end{definition}
\begin{definition}\cite{7}
Let $K\in B(H)$. A sequence $\{x_n\}_{n\geqslant}$ in $H$ is said  to be $K$-Frame for $H$ if there exist two constants $A,B>0$ such that for all $x\in H$,  $$A\|K^*x\|^2\leq \displaystyle{\sum_{n=1}^{+\infty}\vert \langle x,x_n\rangle \vert^2}\leq B\|x\|^2.$$
If there exists $A> 0$ such that $\{x_n\}_{n\geqslant 1}$ satisfies $A\| K^*x\|^2=\displaystyle{\sum_{n=1}^{+\infty}\vert \langle x,x_n\rangle \vert^2}$ for all $x\in H$, we say that $\{x_n\}_{n\geqslant1}$ is a tight $K$-frame for $H$. If $A=1$, $\{x_n\}_{n\geqslant1}$ is called Parseval $K$-frame.\\
\end{definition}
\begin{example}
Let $\{e_n\}_{n\geqslant 1}$ be an orthonormal basis for $H$. Define $K\in B(H)$ such that for all $n\geqslant 1$,  $K(e_n):=e_{n+1}$. $\{K(e_n)\}_{n\geqslant 1}:=\{e_n\}_{n\geqslant 2}$ is not a frame ( it is not a complete sequence). \\
It is clear that for all $x\in H$, $K^*(x)=\displaystyle{\sum_{n=1}^{\infty}\langle x,K(e_n)\rangle e_n}$, then for all $x\in H$, \\$\|K^*(x)\|^2=\displaystyle{\sum_{n=1}^{+\infty}\vert\langle x,K(e_n)\rangle \vert^2}$. Hence $\{K(e_n)\}_{n\geqslant1}=\{e_n\}_{n\geqslant 2}$ is a Parseval $K$-frame for $H$.\\
\end{example}
\begin{remark}\hspace{0.2cm}
\begin{enumerate}
\item[i. ] If $K=I$, then a $K$-frame is just an ordinary frame.
\item[ii. ] An ordinary frame for $H$ is a $K$-frame whatever $K\in B(H)$.
\item[iii.] Let $K\in B(H)$. Every $K$-frame is , in particular, a Bessel sequence, then the synthesis, analysis and frame operators are well defined and bounded. Unlike to the frame case, The synthesis oprator is not surjective, the analysis operator is not injective and the frame operator is not invertible.\\
\end{enumerate}
\end{remark}
The following three results are characterizations of $K$-frames in Hilbert spaces.
\begin{proposition}\cite{11}
Let $\{x_n\}_{n\geqslant1}$ be a Bessel sequence for $H$ whose the frame operator is $S$. Let $K\in B(H)$. The following statements are equivalent:
\begin{enumerate}
\item[i.] $\{x_n\}$ is a $K$-frame.
\item[ii.] There exists a constant $A>0$ such that $AKK^*\leq S$.\\
\end{enumerate}
\end{proposition}

\begin{proposition}\cite{7}\label{Prop1.9}
Let $\{x_n\}_{n\geqslant1}$ be a Bessel sequence for $H$ whose  the synthesis operator is $T$. Let $K\in B(H)$. The following statements are equivalent:
\begin{enumerate}
\item[i.] $\{x_n\}$ is a $K$-frame.
\item[ii.] $R(K)\subset R(T)$.\\
\end{enumerate}
\end{proposition}

\begin{proposition}\cite{7}
Let $\{x_n\}_{n\geqslant1}$ be a Bessel sequence for $H$. Let $K\in B(H)$. The following statements are equivalent:
\begin{enumerate}
\item[i.] $\{x_n\}$ is a $K$-frame.
\item[ii.] There exists a Bessel sequence $\{f_n\}_{n\geqslant1}$ in $H$ such that for all $x\in H$, $$Kx=\displaystyle{\sum_{n=1}^{+\infty}\langle x,f_n\rangle x_n}.$$
Such a Bessel sequence is called a $K$-dual frame to $\{x_n\}_{n\geqslant1}$.\\
\end{enumerate}
\end{proposition}

\begin{remark}
Let $\{x_n\}_{n\geqslant1}$ be a $K$-frame and $\{f_n\}$ be a $K$-dual to $\{x_n\}_{n\geqslant1}$.
\begin{enumerate}
\item[i.] For all $x\in H$, $K^*x=\displaystyle{\sum_{n=1}^{+\infty}\langle x,x_n\rangle f_n}$. Which means that $\{f_n\}_{n\geqslant}$ is a $K^*$-frame.
\item[ii.] $\{x_n\}_{n\geqslant1}$ and $\{f_n\}_{n\geqslant1}$ are interchangeable if and only if $K$ is self-adjoint.\\
\end{enumerate}
\end{remark}
\begin{proposition}\cite{11}\label{Prop1.12}
Let $\{x_n\}_{n\geqslant1}$ be a frame for $H$ and $K\in B(H)$. Then $\{Kx_n\}_{n\geqslant1}$ is a $K$-frame for $H$. \\
\end{proposition}
\begin{definition}\cite{1}
Let $K\in B(H)$. A $K$-frame for $H$ is said to be $K$-minimal frame if its synthesis operator is injective.\\ 
\end{definition}
\begin{remark}
A $K$-minimal frame does not contain zeros. In other words: \\If $\{x_n\}_{n\geqslant 1}$ is a $K$-minimal frame, then for all $n\geqslant 1$, $x_n\neq 0$.\\
\end{remark}
\begin{proposition}\cite{1}
Let $K\in B(H)$ and $\{x_n\}_{n\geqslant1}$ a $K$-frame for $H$. Then the following statements are equivalent:
\begin{enumerate}
\item[i.] $\{x_n\}_{n\geqslant1}$ has a unique $K$-dual frame.
\item[ii.] $\{x_n\}_{n\geqslant1}$ is $K$-minimal frame.\\
\end{enumerate}
\end{proposition}
\begin{definition}\cite{1}
Let $K\in B(H)$ and $\{x_n\}_{n\geqslant}$ a sequence in $H$. $\{x_n\}_{n\geqslant1}$ is said to be $K$-orthonormal basis if: \begin{enumerate}
\item[i.] $\{x_n\}_{n\geqslant1}$ is an orthonormal system in $H$.
\item[ii.] $\{x_n\}_{n\geqslant1}$ is a Parseval $K$-frame.\\
\end{enumerate}
\end{definition}
\begin{theorem}\cite{1}
Let $K\in B(H)$ be an isometry and $\{u_n\}_{n\geqslant1}$ be an orthonormal basis for $H$. Then the following statements are equivalent: 
\begin{enumerate}
\item[i.] $\{x_n\}_{n\geqslant1}$ is a $K$-orthonormal basis for $H$.
\item[ii.] There exists an isometry $L\in B(H)$ such that $R(L)=R(K)$ and $\forall n\geqslant1$, $x_n=Tu_n$.\\
\end{enumerate}
\end{theorem}
\begin{proposition}\cite{1} \label{Prop1.18}
Let $K\in B(H)$ and $\{x_n\}_{n\geqslant 1}$ be a $K$-orthonormal basis for $H$. Then $\{x_n\}_{n\geqslant 1}$ has a unique $K$-dual frame which is exactly $\{K^* x_n\}_{n\geqslant1}$.\\
\end{proposition}
\begin{proposition}\cite{1}\label{Prop1.19}
Let $K\in B(H)$ and $\{x_n\}_{n\geqslant1}$ be a $K$-orthonormal basis. Then the following statements are equivamlent:
\begin{enumerate}
\item[i.] $\{K^*x_n\}_{n\geqslant1}$ is a $K^*$-orthonormal basis for $H$.
\item[ii.] $K$ is co-isometry.\\
\end{enumerate}
\end{proposition}
\section{Main results}
\subsection{$K$-Frames, $L$-Frames and $K\oplus L$-frames}
\medskip
\noindent

In this section, we illustrate the relationship between $K$-frames, $L$-frames and $K\oplus L$-frames.\\ \\
 
The following operators show and explain the relationship between a super Hilbert space and the Hilbert spaces that form it.
\begin{proposition}\label{Prop2.1}
The map $$P_1:\begin{array}{rcl}
H_1\oplus H_2 &\longrightarrow& H_1\oplus H_2\\
x\oplus y &\mapsto& x\oplus0
\end{array},$$
and the map $$P_2:\begin{array}{rcl}
H_1\oplus H_2 &\longrightarrow& H_1\oplus H_2\\
x\oplus y&\mapsto& 0\oplus y
\end{array},$$
are two orthogonal projections on $H_1\oplus H_2$ and $R(P_1)=H_1\oplus 0$ and $R(P_2)=0\oplus H_2$.
\end{proposition}
\begin{proof} 
We have, clearly, $P_1^2=P_1$ and $P_2^2=P_2$. In the other hand, for $x\oplus y, \\a\oplus b \in H_1\oplus H_2$, we have: $$\begin{array}{rcl}
\langle P_1(x\oplus y), a\oplus b\rangle&=&\langle x\oplus0,a\oplus b\rangle \\
&=&\langle x,a\rangle+\langle 0,b\rangle\\
&=& \langle x,a\rangle+\langle y,0\rangle\\
&=& \langle x\oplus y,a\oplus 0\rangle.
\end{array}$$
Hence for all $a\oplus b\in H_1 \oplus H_2$, $P_1^*(a\oplus b)=a\oplus 0=P(a\oplus b)$. Then $P_1^*=P_1$. \\
Similary, we show that $P_2^*=P_2.$ Then $P_1$ and $P_2$ are orthogonal projections.\\  \\
It is clear that $R(P_1)=\{x\oplus 0,\, x\in H_1\}:=H_1\oplus 0$ and $R(P_2)=\{0\oplus y,\, y\in H_2\}:=0\oplus H_2$.\\
\end{proof}
We know that a K-frame is, in particular, a Bessel sequence, which is why it would be useful to start with the following proposition.
\begin{proposition}
Let  $\{x_n\}_{n\geqslant 1}$ and $\{y_n\}_{n\geqslant 1}$ be two sequences in $H_1$ and $H_2$ respectively. The following statements are equivalent:
\begin{enumerate}
\item[i.] $\{x_n\oplus y_n\}_{n\geqslant 1}$ is a Bessel sequence in $H_1 \oplus H_2$.
\item[ii.] $\{x_n\}_{n\geqslant 1}$ and $\{y_n\}_{n\geqslant 1}$ are two Bessel sequences in $H_1$ and $H_2$ respectively.\\
\end{enumerate}
\end{proposition} 
\begin{proof} \hspace{0.2cm}Assume that $\{x_n\oplus y_n\}_{n\geqslant 1}$ is a Bessel sequence in $H_1 \oplus H_2$ with Bessel bound $B$. 
Then $\{P_1(x_n\oplus y_n)\}_{n\geqslant 1}=\{x_n\oplus 0\}_{n\geqslant1}$ is a Bessel sequence for $H_1\oplus 0$ and $B$ is a Bessel bound. That means that for all $x\oplus 0\in H_1\oplus 0$, $\displaystyle{\sum_{n=1}^{+\infty}\vert \langle x\oplus 0, x_n\oplus0\rangle \vert^2}\leq B\|x\oplus0\|^2.$ 
Hence for all $x\in H_1$, $\displaystyle{\sum_{n=1}^{+\infty}\vert\langle x,x_n \rangle \vert^2\leq B\|x\|^2}.$ Then $\{x_n\}_{n\geqslant1}$ is a Bessel sequence in $H_1$.  
By taking $P_2$ instead of $P_1$ we prove, similary, that $\{y_n\}_{n\geqslant1}$ is a Bessel sequence for $H_2$.\\
Conversely, assume that $\{x_n\}_{n\geqslant 1}$ and $\{y_n\}_{n\geqslant 1}$ are two Bessel sequences in $H_1$ and $H_2$ respectively with Bessel bounds $B_1$ and $B_2$ respectively. Note $B=\max\{B_1,B_2\}$. Then for all $x\oplus y\in H_1\oplus H_2$, we have: $$\begin{array}{rcl}
\displaystyle{\sum_{n=1}^{+\infty}\vert \langle x\oplus y,x_n\oplus y_n\rangle\vert^2}&=& \displaystyle{\sum_{n=1}^{+\infty}\vert \langle x,x_n\rangle+\langle y,y_n\rangle \vert^2}\\
&\leq& \displaystyle{\sum_{n=1}^{+\infty}2(\vert \langle x,x_n\rangle \vert^2+\vert\langle y,y_n\rangle \vert^2)}\\
&\leq& 2B_1\|x\|^2+2B_2\|y\|^2\\
&\leq& 2B(\|x\|^2+\|y\|^2)=2B\|x\oplus y\|^2.
\end{array}$$
Hence $\{x_n\oplus y_n\}_{n\geqslant1}$ is a Bessel sequence in $H_1\oplus H_2$.\\ 
\end{proof}
The following result shows the relationship between the  operators associated with $\{x_n\oplus y_n\}_{n\geqslant 1}$ and those associated with $\{x_n\}_{n\geqslant 1}$ and $\{y_n\}_{n\geqslant 1}$ whenever $\{x_n\oplus y_n\}_{n\geqslant 1}$ is a Bessel sequence.
\begin{proposition}\label{Prop2.3}
Let $\{x_n\}_{n\geqslant 1}$ and $\{y_n\}_{n\geqslant 1}$  be two Bessel sequences in $H_1$ and $H_2$ respectively. let $T_1, T_2$ and $T$ be the synthesis operators of $\{x_n\}_{n\geqslant 1}$, $\{y_n\}_{n\geqslant 1}$ and $\{x_n\oplus y_n\}_{n\geqslant1}$ respectively. let $\theta_1, \theta_2$ and $\theta$ be the frame transforms of $\{x_n\}_{n\geqslant 1}$, $\{y_n\}_{n\geqslant 1}$ and $\{x_n\oplus y_n\}_{n\geqslant1}$ respectively. let $S_1, S_2$ and $S$ be the frame operators of $\{x_n\}_{n\geqslant 1}$, $\{y_n\}_{n\geqslant 1}$ and $\{x_n\oplus y_n\}_{n\geqslant1}$ respectively.
 Then: 
\begin{enumerate}
\item[i. ] For all $a\in \ell^2$, $T(a)=T_1(a)\oplus T_2(a)$.
\item[ii. ] For all $x\oplus y\in H_1\oplus H_2$, $\theta(x\oplus y)=\theta_1(x)+\theta_2(y)$.
\item[iii. ] For all $x\oplus y\in H_1\oplus H_2$, $S(x\oplus y)=S_1(x)+T_1\theta_2(y)\oplus S_2(y)+T_2\theta_1(x).$\\
\end{enumerate}
\end{proposition}
\begin{proof}
\begin{enumerate}
\item[i. ] Let $a=\{a_n\}_{n\geqslant1}\in \ell^2$, we have: $$\begin{array}{rcl}
T(a)&=&\displaystyle{\sum_{n=1}^{+\infty}a_n \, x_n\oplus y_n}\\
&=& \displaystyle{\sum_{n=1}^{+\infty}a_n \, x_n} \oplus  \displaystyle{\sum_{n=1}^{+\infty}a_n \, y_n}\\
&=& T_1(a)\oplus T_2(a).
\end{array}$$
\item[ii.] Let $x\oplus y\in H_1\oplus H_2$, we have: $$\begin{array}{rcl}
\theta(x\oplus y)&=&\{\langle x\oplus y,x_n\oplus y_n\rangle \}_{n\geqslant1}\\
&=& \{\langle x,x_n\rangle+\langle y, y_n\rangle \}_{n\geqslant1}\\
&=& \{\langle x,x_n\rangle \}_{n\geqslant1}+\{\langle y,y_n\rangle \}_{n\geqslant 1}\\
&=& \theta_1(x)+\theta_2(y).
\end{array}$$
\item[iii. ] Let $x\oplus y\in H_1\oplus H_2$, we have: $$\begin{array}{rcl}
S(x\oplus y)&=& T\theta(x\oplus y)\\
&=& T(\theta_1(x)+\theta_2(y))\\
&=& T(\theta_1(x))+T(\theta_2(y))\\
&=& T_1(\theta_1(x))\oplus T_2(\theta_1(x))+ T_1(\theta_2(y))\oplus T_2(\theta_2(y))\\
&=& S_1(x)\oplus T_2\theta_1(x)+ T_1\theta_2(y)\oplus S_2(y)\\
&=& S_1(x)+T_1\theta_2 y\oplus S_2(y)+T_2\theta_1(x).
\end{array}$$
\end{enumerate}
\end{proof}
The following result is a necessary condition for a sequence to be M-frame in $H_1\oplus H_2$.
\begin{proposition}\label{Prop2.4}
Let $M\in B(H_1\oplus H_2)$ and $\{x_n\oplus y_n\}_{n\geqslant 1}$ be a sequence in $H_1\oplus H_2$. If $\{x_n\oplus y_n\}_{n\geqslant 1}$ is a $M$-frame for $H_1\oplus H_2$ with $M$-frame bounds $A$ and $B$, then: 
\begin{enumerate}
\item[i. ] For all $x\in H_1$,  $A\| M_1^*(x)\|^2\leq \displaystyle{\sum_{n=1}^{+\infty}\vert\langle x,x_n\rangle\vert^2}\leq B\| x\|^2$.
\item[ii. ]  For all $y\in H_2$,  $A\| M_2^*(y)\|^2\leq \displaystyle{\sum_{n=1}^{+\infty}\vert\langle y,y_n\rangle \vert^2}\leq B\| y\|^2$.
\end{enumerate}
Where $M_1: H_1\oplus H_2\longrightarrow H_1 $ and $M_2:H_1\oplus H_2\longrightarrow H_2$ are the bounded linear operators such that for all $x\oplus y\in H_1\oplus H_2$, $M(x\oplus y)=M_1(x\oplus y)\oplus M_2(x\oplus y).$\\
\end{proposition}
\begin{proof}
We have for all $x\oplus y\in H_1\oplus H_2$, $$A\|M^*(x\oplus y)\|^2\leq \displaystyle{\sum_{n=1}^{+\infty}\vert\langle x\oplus y,x_n\oplus y_n\rangle \vert^2}\leq B\|x\oplus y\|^2.$$
\begin{enumerate}
\item[i. ] In particular for $y=0$, we have for all $x\in H_1$, $$A\|M^*(x\oplus 0)\|^2\leq \displaystyle{\sum_{n=1}^{+\infty}\vert\langle x,x_n\rangle \vert^2}\leq B\|x\|^2.$$
Let $a\oplus b\in H_1\oplus H_2$ and let $x\in H_1$, we have: $$\begin{array}{rcl}
\langle M(a\oplus b),x\oplus 0\rangle &=&\langle M_1(a\oplus b)\oplus M_2(a\oplus b), x\oplus 0\rangle\\
&=& \langle M_1(a \oplus b), x\rangle \\
&=& \langle a\oplus b,M_1^*(x)\rangle.
\end{array}$$
Hence $M^*(x\oplus 0)=M_1^*(x)$. Then  For all $x\in H_1$,  $$A\| M_1^*(x)\|^2\leq \displaystyle{\sum_{n=1}^{+\infty}\vert\langle x,x_n\rangle \vert^2}\leq B\| x\|^2.$$
\item[ii. ] In particular for $x=0$, we have for all $y\in H_2$, $$A\|M^*(0\oplus y)\|^2\leq \displaystyle{\sum_{n=1}^{+\infty}\vert\langle y,y_n\rangle \vert^2}\leq B\|y\|^2.$$
Let $a\oplus b\in H_1\oplus H_2$ and let $y\in H_2$, we have: $$\begin{array}{rcl}
\langle M(a\oplus b),0\oplus y\rangle&=&\langle M_1(a\oplus b)\oplus M_2(a\oplus b), 0\oplus y\rangle\\
&=& \langle M_2(a \oplus b), y\rangle\\
&=& \langle a\oplus b,M_2^*(y)\rangle.
\end{array}$$
Hence $M^*(0\oplus y)=M_2^*(y)$. Then  For all $x\in H_1$,  $$A\| M_2^*(y)\|^2\leq \displaystyle{\sum_{n=1}^{+\infty}\vert\langle y,y_n\rangle \vert^2}\leq B\| y\|^2.$$
\end{enumerate}
\end{proof}
The following corollary justifies why we will always assume  $\{x_n\}_{n\geqslant1}$ and $\{y_n\}_{n\geqslant 1}$ to be $K$-frame and $L$-frame respectively.
\begin{corollary}\label{Cor2.5}
Let $K\in B(H_1)$ and $L\in B(H_2)$. If a sequence $\{x_n\oplus y_n\}_{n\geqslant 1}$ in $H_1\oplus H_2$ is a $K\oplus L$-frame for $H_1\oplus H_2$, then $\{x_n\}_{n\geqslant 1}$ is a $K$-frame for $H_1$ and $\{y_n\}_{n\geqslant 1}$ is a $L$-frame for $H_2$.\\
\end{corollary}
\begin{proof}
We have for all $x\oplus y\in H_1\oplus H_2$, $K\oplus L(x\oplus y):=K(x)\oplus L(y)=M_1(x\oplus y)\oplus M_2(x\oplus y).$ Let $a\oplus b\in H_1\oplus H_2$ and $x\in H_1$, we have: $$\begin{array}{rcl}
\langle M_1(a\oplus b),x\rangle &=& \langle K(a),x\rangle \\
&=& \langle a,K^*(x)\rangle\\
&=& \langle a\oplus b,K^*(x)\oplus 0\rangle.
\end{array}$$ Hence $M_1^*(x)=K^*(x)\oplus 0$, then $\| M_1^*(x)\|=\|K^*(x)\|$.\\
Let $a\oplus b\in H_1\oplus H_2$ and $y\in H_2$, we have: $$\begin{array}{rcl}
\langle M_2(a\oplus b),y\rangle &=& \langle L(b),y\rangle \\
&=& \langle b,L^*(y)\rangle \\
&=& \langle a\oplus b,0\oplus L^*(y)\rangle.
\end{array}$$ Hence $M_2^*(y)=0\oplus L^*(y)$, then $\| M_2^*(y)\|=\|L^*(y)\|$.\\
Proposition \ref{Prop2.4} completes the proof.
\end{proof}
\begin{corollary}
Let $K\in B(H_1)$ and $L\in B(H_2)$. Then there exist a $K$-frame $\{x_n\}_{n\geqslant 1}$ for $H_1$ and a $L$-frame $\{y_n\}_{n\geqslant 1}$ for $H_2$ such that $\{x_n\oplus y_n\}_{n\geqslant 1}$ is a $K\oplus L$-frame for $H_1\oplus H_2$.\\
\end{corollary}
\begin{proof}
Take any frame $\{x_n\oplus y_n\}_{n\geqslant 1}$ for $H_1\oplus H_2$. By proposition \ref{Prop1.12}, $\{K\oplus L(x_n\oplus y_n)\}_{n\geqslant 1}$ is a $K\oplus L$-frame for $H_1\oplus H_2$. That means that $\{K(x_n)\oplus L(y_n)\}_{n\geqslant 1}$ is a $K\oplus L$-frame for $H_1\oplus H_2$. And by  corollary \ref{Cor2.5}, $\{K(x_n)\}_{n\geqslant 1}$ is a $K$-frame for $H_1$ and $\{L(y_n)\}_{n\geqslant 1}$ is a $L$-frame for $H_2$.\\
\end{proof}
This lemma is very useful for the rest.
\begin{lemma}
Let $K\in B(H_1)$ and $L\in B(H_2)$. Then $(K\oplus L)^*=K^*\oplus L^*$.\\
\end{lemma}
\begin{proof}
Let $x\oplus y, a\oplus b \in H_1\oplus H_2$, we have: $$\begin{array}{rcl}
\langle K\oplus L(x\oplus y), a\oplus b\rangle &=& \langle K(x)\oplus L(y),a\oplus b\rangle\\
&=& \langle K(x),a\rangle +\langle L(y),b\rangle\\
&=& \langle x,K^*(a)\rangle +\langle y,L^*(b)\rangle \\
&=& \langle x\oplus y, K^*(a)\oplus L^*(b)\rangle\\
&=& \langle x\oplus y,(K^*\oplus L^*)(a \oplus b)\rangle .
\end{array}$$\\
\end{proof}
The following proposition shows that there is no $K\oplus L$-frame for $H\oplus H$ in the form $\{x_n\oplus x_n\}_{n\geqslant 1}$ whenever $K,L\neq 0$.
\begin{proposition}\label{Prop2.8}
Let $K, L\in B(H)$ and $\{x_n\oplus y_n\}_{n\geqslant1}$ be a Bessel sequence for $H\oplus H$. Then: 
$$\{x_n\oplus x_n\}_{n\geqslant 1} \text{ is a $K\oplus L$-frame for } H\oplus H \Longleftrightarrow K=L=0.$$\\ \end{proposition}
\begin{proof}\hspace{0.2cm}
\begin{enumerate}
\item[i.] Assume that $K\neq 0$ or $L\neq 0$. We can suppose the case of $K\neq 0$ and the other case will be obtained similary.\\
For $x\in H$ such that $K^*(x)\neq 0$, we have $\;\|K^*\oplus L^*(x\oplus (-x)\,)\|^2\\=\| K^*(x)\|^2+\|L^*(x)\|^2\geqslant \|K^*(x)\|^2\neq 0$ and $\displaystyle{\sum_{n=1}^{+\infty}\vert \langle x\oplus (-x),x_n\oplus x_n\rangle\vert^2=0}$.\\ Hence $\{x_n\oplus x_n\}_{n\geqslant 1}$ is not a $K\oplus L$-frame for $H_1\oplus H_2$.\\
\item[ii. ] Assume that $K=L=0$. the result follows immediately from the fact that $\{x_n\oplus y_n\}_{n\geqslant 1}$ is a Bessel sequence.\\
\end{enumerate}
\end{proof}
Two non-disjoints frames (two frames whose the direct sum is not a frame for the super Hilbert space of the underlying Hilbert spaces) shows that if $\{x_n\}_{n\geqslant 1}$ is a $K$-frame and $\{y_n\}_{n\geqslant 1}$ is a $L$-frame, $\{x_n\oplus y_n\}_{n\geqslant 1}$ is not necessarly a $K\oplus L$-frame.\\ \\ Another simple example to see that is the following example which is an immediate result of proposition \ref{Prop2.8}. 
\begin{example}
Let $K\in B(H)$ such that $K\neq 0$. Take any $K$-frame $\{x_n\}_{n\geqslant 1}$ for $H$,  $\{x_n\oplus x_n\}_{n\geqslant 1}$ is not a $K\oplus K$-frame for the super Hilbert space $H\oplus H$. \\
\end{example}

The following result is a sufficient condition for a sequence $\{x_n\oplus y_n\}_{n\geqslant 1}$ in $H_1\oplus H_2$ to be a $K\oplus L$-frame.
\begin{proposition}\label{Prop2.10}
Let $K\in B(H_1)$ and $L\in B(H_2)$. 
Let $\{x_n\}_{n\geqslant 1}$ be a $K$-frame for $H_1$ and $\{y_n\}_{n\geqslant 1}$ be a $L$-frame for $H_2$. 
Let $\theta_1$ and $\theta_2$ be the frame tronsforms of $\{x_n\}_{n\geqslant 1}$ and $\{y_n\}_{n\geqslant 1}$ respectively.\\ If $R(\theta_1) \perp R(\theta_2)$ in $\ell^2$, then $\{x_n\oplus y_n\}_{n\geqslant 1}$ is a $K\oplus L$-frame for $H_1\oplus H_2$.\\
\end{proposition}
\begin{proof}
Denote by $A_1$ and $A_2$ the lower bounds for $\{x_n\}_{n\geqslant 1}$ and $\{y_n\}_{n\geqslant 1}$ respectively, and let $A=\min\{A_1, A_2\}$. 
We have for all $x\in H_1, y\in H_2$,  $\langle \theta_1(x),\theta_2 (y)\rangle =0$, then for all $x\in H_1, y\in H_2$, $\langle T_2\theta_1(x),y\rangle =0$ and $\langle T_1\theta_2(x),y\rangle =0$ where $T_1$ and $T_2$ are the synthesis operators for $\{x_n\}_{n\geqslant 1}$ and $\{y_n\}_{n\geqslant 1}$ respectively. Hence $T_2\theta_1=T_1\theta_2=0$.
Hence, by proposition \ref{Prop2.3}, the frame operator $S$ of $\{x_n\oplus y_n\}_{n\geqslant 1}$ is defined by $S(x\oplus y)=S_1(x)\oplus S_2(y)=(S_1\oplus S_2)(x\oplus y)$; where $S_1$ and $S_2$ are the frame operators for $\{x_n\}_{n\geqslant 1}$ and $\{y_n\}_{n\geqslant 1}$ respectively. \\
Then $$\begin{array}{rcl}
\langle S(x\oplus y),x\oplus y\rangle &=&\langle S_1(x),x\rangle +\langle S_2(y),y\rangle \\
&\geqslant& A_1\|K^*(x)\|^2+A_2\|L^*(y)\|^2\\
&\geqslant& A\|K^*(x)\oplus L^*(y)\|^2\\
&\geqslant& A\| (K\oplus L)^*(x\oplus y)\|^2.
\end{array}.$$
Hence $\{x_n\oplus y_n\}_{n\geqslant 1}$ is a $K\oplus L$-frame for $H_1\oplus H_2$.\\ 
\end{proof}
The following example is an example of a $K\oplus L$-frame in a super Hilbert space.
\begin{example}
Let $H$ be a Hilbert space in which $\{e_n\}_{n\geqslant 1}$ is an orthonormal basis. Let $M=\overline{span}\{e_{4n},\; n\geqslant 1\}$ and $N=\overline{span}\{e_{2n-1},\; n\geqslant 1\}$. Let $P$ and $Q$ be the orthogonal projections from $H$ onto $M$ and $N$ respectively. Then $\{P(e_n)\oplus Q(e_n)\}_{n\geqslant1}$ is a $P\oplus Q$-frame for $H\oplus H$. In fact, $\{P(e_n)\}_{n\geqslant 1}$ and $\{Q(e_n)\}_{n\geqslant 1}$ are $P$-frame and $Q$-frame, respectively, for $H$. Let $\theta_1$ and $\theta_2$ be their frame transforms  respectively. Then for all $x,y\in H$, we have: $$\begin{array}{rcl}
\langle\theta_1(x),\theta_2(y)\rangle&=& \displaystyle{\sum_{n=1}^{\infty}\langle x,P(e_n)\rangle\,\overline{\langle y,Q(e_n)\rangle}}\\
&=& \displaystyle{\sum_{n=1}^{\infty} \langle x,P(e_n)\rangle\,\langle Q(e_n),y\rangle}\\
&=& \displaystyle{\sum_{n=1}^{\infty}\langle Px,e_n\rangle\,\langle e_n,Qy\rangle}\\
&=& \langle Px,Qy\rangle \\
&=& 0 \hspace{2cm} (\text{ Since } M\perp N).
\end{array}$$
Then $R(\theta_1)\perp R(\theta_2)$. Hence, by proposition \ref{Prop2.10}, $\{P(e_n)\oplus Q(e_n)\}_{n\geqslant1}$ is a $P\oplus Q$-frame for $H\oplus H$.\\ 
\end{example}

The following result is another necessary condition for $\{x_n\oplus y_n\}_{n\geqslant1}$ to be $K\oplus L$-frame for the super Hilbert space.\
\begin{proposition}\label{Prop2.16}
Let $K\in B(H_1)$ and $L\in B(H_2)$. Let $\{x_n\}_{n\geqslant 1}$ be a $K$-frame for $H_1$ and $\{y_n\}_{n\geqslant 1}$ be a $L$-frame for $H_2$.
If $\{x_n \oplus y_n\}_{n\geqslant 1}$ is a $K\oplus L$-frame for $H_1\oplus H_2$, then:$$\left\lbrace
\begin{array}{rcl}
R(K) &\subset& T_1(\, N(T_2)\,),\\
R(L)&\subset& T_2(\,N(T_1)\,).
\end{array} 
\right.$$
Where $T_1$ and $T_2$ are the synthesis operators for $\{x_n\}_{n\geqslant 1}$ and $\{y_n\}_{n\geqslant 1}$ respectively.\\
\end{proposition}
\begin{proof} Assume that $\{x_n \oplus y_n\}_{n\geqslant 1}$ is a $K\oplus L$-frame for $H_1\oplus H_2$ and let $T$  its synthesis operator. 
By proposition \ref{Prop1.9}, $R(K\oplus L)\subset R(T)$. Then for all $x\oplus y\in H_1\oplus H_2$, there exists $a\in \ell^2$, such that $K\oplus L(x\oplus y)=T(a)$. That means, by proposition \ref{Prop2.3}, that for all $x\oplus y\in H_1\oplus H_2$, there exists $a\in \ell^2$ such that $K(x)\oplus L(y)=T_1(a)\oplus T_2(a)$.\\
By taking $y=0$, there exists $a\in \ell^2$, such that $K(x)=T_1(a)$ and $0=L(y)=T_2(a)$. Hence $K(x)=T_1(a)$ and $a\in N(T_2)$. Hence $R(K)\subset T_1(\,N(T_2)\,)$.\\
By taking $x=0$, there exists $a\in \ell^2$ such that $0=K(x)=T_1(a)$ and $L(y)=T_2(a)$. Hence $L(y)=T_2(a)$ and $a\in N(T_1)$. Hence $R(L)\subset T_2(\, N(T_1)\,)$.\\
\end{proof}
The following result shows that the non-minimality of the two sequences $\{x_n\}_{n\geqslant 1}$ and $\{y_n\}_{n\geqslant 1}$ is necessary for their direct sum to be a $K\oplus L$-frame whenever $K,L\neq 0$.
\begin{corollary}\label{Cor2.17}
Let $K\in B(H_1)$ and $L\in B(H_2)$. Let $\{x_n\oplus y_n\}_{n\geqslant 1}$ be a $K\oplus L$-frame for $H_1\oplus H_2$. Then: 
\begin{enumerate}
\item[i. ] If $\{x_n\}_{n\geqslant 1}$ is $K$-minimal, then $L=0$.
\item[ii. ] If $\{y_n\}_{n\geqslant 1}$ is $L$-minimal, then $K=0$.
\item[iii. ] If $\{x_n\}_{n\geqslant 1}$ is $K$-minimal and $\{y_n\}_{n\geqslant 1}$ is $L$-minimal, then $K=0$ and $L=0$.\\
\end{enumerate}
In particular: Let $K\neq 0$ and $L\neq 0$. If at least one of $\{x_n\}_{n\geqslant 1}$ and $\{y_n\}_{n\geqslant 1}$ is minimal, then $\{x_n\oplus y_n\}_{n\geqslant 1}$ is never a $K\oplus L$-frame for the super Hilbert space $H_1\oplus H_2$.\\
\end{corollary}

\begin{proof}\hspace{0.2cm}
\begin{enumerate}
\item[i. ] Assume that $\{x_n\}_{n\geqslant 1}$ is $K$-minimal, then $N(T_1)=(0)$. By proposition \ref{Prop2.16}, we  obtain $R(L)\subset (0)$. Hence $L=0$.\\
\item[ii. ] Assume that $\{y_n\}_{n\geqslant 1}$ is $L$-minimal, then $N(T_2)=(0)$. By proposition \ref{Prop2.16}, we have that $R(K)\subset (0)$. hence $K=0$.\\
\item[iii. ] Direct result of i. and ii.
\end{enumerate}
\end{proof}
\begin{example}
Let $\{e_n\}_{n\geqslant 1}$ and $\{f_n\}_{n\geqslant 1}$ be  orthonormal bases for $H_1$ and $H_2$ respectively. Let $K\in B(H_1)$ such that $K(e_n)=e_{2n}$ and $L\in B(H_2)$ such that $L(f_n)=f_{2n-1}$ ($\forall n\geqslant 1$). If we take $\{x_n\}_{n\geqslant 1}:=\{e_{2n}\}_{n\geqslant 1}$ and $\{y_n\}_{n\geqslant 1}:=\{f_{2n-1}\}_{n\geqslant 1}$. $\{x_n\}_{n\geqslant 1}$ and $\{y_n\}_{n\geqslant 1}$ are $K$-frame for $H_1$ and $L$-frame for $H_2$ respectively. But $\{x_n\oplus y_n\}_{n\geqslant 1}$ is not a $K\oplus L$-frame for $H_1\oplus H_2$. In fact, it suffices to constate that $\{x_n\}_{n\geqslant 1}:=\{e_{2n}\}_{n\geqslant 1}$ is minimal.\\

\end{example}
\begin{corollary}
If $\{x_n\}_{n\geqslant 1}$ is a Riesz basis for a Hilbert space $H_1$, then there is no sequence $\{y_n\}_{n\geqslant1}$ in any Hilbert space $H_2$ such that $\{x_n\oplus y_n\}_{n\geqslant 1}$ is a frame for $H_1\oplus H_2$.\\
\end{corollary}
\begin{proof}
In this case $K$ and $L$ are the identities of the two Hilbert spaces $H_1$ and $H_2$ respectively. Corollary \ref{Cor2.17} completes the proof.\\
\end{proof}

\subsection{$K$-duaity, $L$-duality and $K\oplus L$-duality}
\medskip
\noindent

In this section, we show the relationship between  $K$-duality, $L$-duality and $K\oplus L$-duality.
\begin{proposition}\label{Prop2.20}
Let $K\in B(H_1)$ and $L\in B(H_2)$. \\Let $\{x_n \oplus y_n\}_{n\geqslant 1}$ be a $K\oplus L$-frame for $H_1\oplus H_2$. $\{a_n\oplus b_n\}_{n\geqslant 1}$ is a $K\oplus L$-dual frame to $\{x_n\oplus y_n\}_{n\geqslant 1}$, then $\{a_n\}_{n\geqslant 1}$ is a $K$-dual frame to $\{x_n\}_{n\geqslant 1}$ and $\{y_n\}_{n\geqslant 1}$ is a $L$-dual frame to  $\{y_n\}_{n\geqslant 1}$.\\
\end{proposition}
\begin{proof}
Let $a_n\oplus b_n\in H_1\oplus H_2$ be a $K\oplus L$-dual frame to  $\{x_n\oplus y_n\}_{n\geqslant 1}$. Then for all $x\oplus y\in H_1\oplus H_2$, $$\begin{array}{rcl}
K\oplus L(x\oplus y)&=&\displaystyle{\sum_{n=1}^{+\infty} \langle x\oplus y, a_n\oplus b_n\rangle x_n\oplus y_n}\\
&=& \displaystyle{\sum_{n=1}^{+\infty} \langle x\oplus y,a_n\oplus b_n\rangle x_n}\oplus \displaystyle{\sum_{n=1}^{+\infty} \langle x\oplus y,a_n\oplus b_n\rangle y_n}.\\
\end{array}$$
Thus for all $x\in H_1$ and $y\in H_2$, $$\left\lbrace
\begin{array}{rcl}
K(x)&=&\displaystyle{\sum_{n=1}^{+\infty} \langle x\oplus y,a_n\oplus b_n\rangle x_n}.\\
L(y)&=&\displaystyle{\sum_{n=1}^{+\infty} \langle x\oplus y,a_n\oplus b_n\rangle y_n}.
\end{array}
\right.$$
In particular, by taking $y=0$ in the first equality and $x=0$ in the second one, we obtain: 
$$\left\lbrace
\begin{array}{rcl}
K(x)&=&\displaystyle{\sum_{n=1}^{+\infty} \langle x,a_n\rangle x_n},\\
L(y)&=&\displaystyle{\sum_{n=1}^{+\infty} \langle y,b_n\rangle y_n}.
\end{array}
\right.$$
Hence $\{a_n\}_{n\geqslant 1}$ is a $K$-dual frame to $\{x_n\}_{n\geqslant 1}$ and $\{b_n\}_{n\geqslant 1}$ is a $L$-dual frame to $\{y_n\}_{n\geqslant 1}$.\\
\end{proof}
\vspace{0.4cm}
One can wonder whether the converse of the above proposition holds. The following result shows that it does not always hold and gives a necessary and sufficient condition under which it holds. 

\begin{proposition}
Let $K\in B(H_1)$ and $L\in B(H_2)$. Let $\{x_n\}_{n\geqslant 1}$ be a $K$-frame for $H_1$ whose $\{f_n\}_{n\geqslant 1}$ is a $K$-dual frame and   $\{y_n\}_{n\geqslant 1}$ be a $L$-frame for $H_2$ whose $\{g_n\}_{n\geqslant 1}$ is a $L$-dual frame. Then the following statements are equivalent:
\begin{enumerate}
\item[i. ] $\{x_n\oplus y_n\}_{n\geqslant 1}$ is a $K\oplus L$-frame for $H_1\oplus H_2$ whose $\{f_n\oplus g_n\}_{n\geqslant 1}$ is a $K\oplus L$-dual frame.
\item[ii. ] $T_2\theta_1=0_{H_1}$ and $T_1\theta_2=0_{H_2}$.
\end{enumerate}
Where $T_1$ and $T_2$ are the synthesis operators of $\{x_n\}_{n\geqslant 1}$ and $\{y_n\}_{n \geqslant 1}$ respectively and $\theta_1$ and $\theta_2$ are the frame transforms of $\{f_n\}_{n\geqslant 1}$ and $\{g_n\}_{n\geqslant 1}$ repectively.\\
\end{proposition}
\begin{proof}\hspace{0.2cm}
\begin{enumerate}
\item[i $\Longrightarrow $i.] Assume that $\{x_n\oplus y_n\}_{n\geqslant 1}$ is a $K\oplus L$-frame for $H_1\oplus H_2$ whose $\{f_n\oplus g_n\}_{n\geqslant 1}$ is a $K\oplus L$-dual frame.\\
Then for all $x\oplus y\in H_1\oplus H_2$, we have: $$K\oplus L(x\oplus y)=\displaystyle{\sum_{n=1}^{+\infty}\langle x\oplus y, f_n\oplus g_n\rangle x_n\oplus y_n.}$$
Hence for all $x\oplus y\in H_1\oplus H_2$, $$K(x)\oplus L(y)=\displaystyle{\sum_{n=1}^{\infty} \langle x,f_n\rangle x_n}+\displaystyle{\sum_{n=1}^{\infty} \langle y,g_n\rangle x_n  }\oplus \displaystyle{\sum_{n=1}^{\infty}\langle x,f_n\rangle y_n   }+\displaystyle{\sum_{n=1}^{\infty}\langle y,g_n\rangle y_n }.$$
That means that for all $x\oplus y\in H_1\oplus H_2$, $$K(x)\oplus L(y)=K(x)+\displaystyle{\sum_{n=1}^{\infty} \langle y,g_n\rangle x_n  }\oplus \displaystyle{\sum_{n=1}^{\infty}\langle x,f_n\rangle y_n   }+L(y).$$
Then for all $x\oplus y\in H_1\oplus H_2$, $$\left\lbrace
\begin{array}{rcl}
\displaystyle{\sum_{n=1}^{\infty} \langle y,g_n\rangle x_n  }&=&0,\\
\displaystyle{\sum_{n=1}^{\infty}\langle x,f_n\rangle y_n   }&=&0.
\end{array}
\right.$$
Thus for all  $ x\oplus y \in H_1\oplus H_2$, $T_1\theta_2(y)=0$ and $T_2\theta_1(x)=0$.\\
\item[ii $\Longrightarrow $ i. ] Assume that $T_2\theta_1=0_{H_1}$ and $T_1\theta_2=0_{H_2}$. Then for all $x\oplus y\in H_1\oplus H_2$, $$\left\lbrace
\begin{array}{rcl}
\displaystyle{\sum_{n=1}^{\infty}\langle x,f_n\rangle y_n   }&=&0\\
\displaystyle{\sum_{n=1}^{\infty} \langle y,g_n\rangle x_n  }&=&0.
\end{array}
\right.$$
In the other hand, we have for all $x\oplus y\in H_1\oplus H_2$, $$\begin{array}{rcl}
K\oplus L(x\oplus y)&=& K(x)\oplus L(y)\\
&=& \displaystyle{\sum_{n=1}^{+\infty} \langle x,f_n\rangle x_n }\oplus \displaystyle{\sum_{n=1}^{+\infty}\langle y,g_n\rangle y_n  }\\
&=& \displaystyle{\sum_{n=1}^{+\infty} \langle x,f_n\rangle x_n } +\displaystyle{\sum_{n=1}^{\infty} \langle y,g_n\rangle x_n  } \oplus \displaystyle{\sum_{n=1}^{\infty}\langle x,f_n\rangle y_n}+ \displaystyle{\sum_{n=1}^{+\infty}\langle y,g_n\rangle y_n}\\
&=& \displaystyle{\sum_{n=1}^{+\infty} \langle x\oplus y,f_n\oplus g_n\rangle x_n}\oplus \displaystyle{\sum_{n=1}^{+\infty}\langle x\oplus y,f_n\oplus g_n\rangle y_n}\\
&=&\displaystyle{\sum_{n=1}^{+\infty} \langle x\oplus y,f_n\oplus g_n\rangle x_n\oplus y_n}.
\end{array}$$
Hence $\{x_n\oplus y_n\}_{n\geqslant 1}$ is a $K\oplus L$-frame for $H_1\oplus H_2$ whose $\{f_n\oplus g_n\}_{n\geqslant 1}$ is a $K\oplus L$-dual frame.\\ 
\end{enumerate}
\end{proof}

\subsection{$K\oplus L$-minimal frames}
\medskip
\noindent

In what follows, we investigate $K\oplus L$-minimal frames for super Hilbert spaces.\\ \\

The following result gives a necessary and sufficient condition for a $M$-frame in $H_1\oplus H_2$ to be a $M$-minimal frame.
\begin{proposition}\label{Prop2.12}
Let $M\in B(H_1\oplus H_2)$. Let $\{x_n\oplus y_n\}_{n\geqslant 1}$ be a $M$-frame for $H_1\oplus H_2$, then the following statements are equivalent:
\begin{enumerate}
\item[i. ] $\{x_n\oplus y_n\}_{n\geqslant 1}$ is a $M$-minimal frame for $H_1\oplus H_2$.
\item[ii. ]  $N(T_1)\cap N(T_2)=(0).$
\end{enumerate}
where $T_1$ and $T_2$ are the synthesis operator of $\{x_n\}_{n\geqslant 1}$ and $\{y_n\}_{n \geqslant1}$ respectively.\\
\end{proposition}
\begin{proof}
By proposition \ref{Prop2.3}, for all $a\in \ell^2$, $T(a)=T_1(a)\oplus T_2(a)$ where $T$ is the synthesis operator of $\{x_n\oplus y_n\}_{n\geqslant1}$. Then $N(T)=N(T_1)\cap N(T_2)$. Hence:\\$\{x_n\oplus y_n\}_{n\geqslant 1}$ is a $M$-minimal frame $ \Longleftrightarrow    N(T)=(0)\Longleftrightarrow N(T_1)\cap N(T_2)=(0).$\\
\end{proof}
The following result gives a sufficient condition for $\{x_n\oplus y_n\}_{n\geqslant 1}$ to be a $K\oplus L$-minimal frame for the super Hilbert space.\\ \\

We will need the following lemma.
\begin{lemma}\label{Lem2.13}
Let $A,B\subset H$ where $H$ is a Hilbert space. Then the following statements are equivalent:\\
\begin{enumerate}
\item[i. ] $\overline{A}=B^\perp$.\\
\item[ii. ] $\left\lbrace
\begin{array}{rcl}
A \,&\perp&\, B\\
A^\perp \,&\cap&\, B^\perp=(0)
\end{array}
\right.$
\end{enumerate}
\end{lemma}
\begin{proof}\hspace{0.2cm}
\begin{enumerate}
\item[i$\Longrightarrow$ ii.] Assume that $\overline{A}=B^\perp$. Then $A\subset \overline{A}=B^\perp$. Hence $A\perp B$.\\
In the other hand, we have $A^\perp \cap B^\perp= \overline{A}^\perp \cap B^\perp=\overline{A}^\perp \cap \overline{A}=(0)$.\\
\item[ii$\Longrightarrow$ i. ] Assume that $\left\lbrace
\begin{array}{rcl}
A \,&\perp&\, B\\
A^\perp \,&\cap&\, B^\perp=(0).
\end{array}
\right.$\\
Since $A\,\perp \, B$, then $A\subset B^\perp$, hence $\overline{A}\subset B^\perp$. In the other hand, let $x\in B^\perp$, since $H=\overline{A}\oplus^\perp  \overline{A}^\perp=\overline{A}\oplus^\perp  A^\perp$. Then there exist $a\in \overline{A}$ and $b\in A^\perp$ such that $x=a\oplus^\perp \, b$. We have $b\in A^\perp$ and $b=x-a$ with $x\in B^\perp$ and $a\in \overline{A}\subset B^\perp$, then $b\in B^\perp$. Thus $b\in A^\perp \cap B^\perp=(0)$. Hence $x=a\in \overline{A}$. That means that $\overline{A}=B^\perp$.
\end{enumerate}
\end{proof}
\begin{proposition}\label{Prop2.14}
Let $K\in B(H_1)$ and $L\in B(H_2)$. 
Let $\{x_n\}_{n\geqslant 1}$ be a $K$-frame for $H_1$ and $\{y_n\}_{n\geqslant 1}$ be a $L$-frame for $H_2$. 
Let $\theta_1$ and $\theta_2$ be the frame tronsforms of $\{x_n\}_{n\geqslant 1}$ and $\{y_n\}_{n\geqslant 1}$ respectively.\\ If $\overline{R(\theta_1)}=R(\theta_2)^\perp$, then $\{x_n\oplus y_n\}_{n\geqslant 1}$ is a $K\oplus L$-minimal frame for $H_1\oplus H_2$.\\
\end{proposition}
\begin{proof}
By Lemma \ref{Lem2.13}, we have $R(\theta_1)\perp  R(\theta_2)$ and $R(\theta_1)^\perp \, \cap \, R(\theta_2)^\perp=(0)$.\\
Proposition \ref{Prop2.10} implies that $\{x_n\oplus y_n\}_{n\geqslant 1}$ is a $K\oplus L$-frame for $H_1\oplus H_2$. Since $N(\theta_1^*)=R(\theta_1)^\perp$ and $N(\theta_2^*)=R(\theta_2)^\perp$, then  $N(\theta_1^*)\cap N(\theta_2^*)=(0)$, that means that $N(T_1)\cap N(T_2)=(0)$ where $T_1$ and $T_2$ are the synthesis operators of $\{x_n\}_{n\geqslant 1}$ and $\{y_n\}_{n\geqslant 1}$ respectively. Hence, by proposition \ref{Prop2.12},  $\{x_n\oplus y_n\}_{n\geqslant 1}$ is a $K\oplus L$-minimal frame for $H_1\oplus H_2$. \\
\end{proof}
The following example is an example of a $K\oplus L$-minimal frame for a super Hilbert space.
\begin{example}\label{Ex2.15}\hspace{0.2cm}
Let $H_1$ and $H_2$ be two Hilbert spaces in which $\{e_n\}_{n\geqslant 1}$ and $\{f_n\}_{n\geqslant1}$ are orthonormal bases respectively.\\ \\
Let $\{x_n\}_{n\geqslant}\subset H_1$ be the sequence defined as follows: $\forall n\geqslant 1$, $x_n=\left\lbrace
\begin{array}{rcl}
&0& \text{ if } n \text{ is odd,}\\
&e_{n} &\text{ if } n \text{ is even.}
\end{array}
\right.$\\ \\
Let $\{y_n\}_{n\geqslant}\subset H_2$ be the sequence defined as follows: $\forall n\geqslant 1$, $y_n=\left\lbrace
\begin{array}{rcl}
&0& \text{ if } n \text{ is even,}\\
&f_{n} &\text{ if } n \text{ is odd.}
\end{array}
\right.$\\ \\
Let $K\in B(H_2)$ such that for all $n\geqslant1$, $K(e_n)=e_{2n}$ and $L\in B(H_2)$ such that for all $n\geqslant 1$, $L(f_n)=f_{2n-1}$. Then for all $x\in H_1$ , we have: $$
\begin{array}{rcl}
\|K^*(x)\|^2&=&\displaystyle{\sum_{n=1}^{\infty}\vert\langle x,K(e_n)\rangle\vert^2}\\
&=&\displaystyle{\sum_{n=1}^{\infty}\vert\langle x,e_{2n}\rangle \vert^2}\\
&=&\displaystyle{\sum_{n=1}^{\infty}\vert \langle x,x_n\rangle \vert^2}.\\
\end{array}$$
And  for all $y\in H_2$, we have: $$
\begin{array}{rcl}
\|L^*(y)\|^2&=&\displaystyle{\sum_{n=1}^{\infty}\vert \langle y,L(f_n)\rangle \vert^2}\\
&=&\displaystyle{\sum_{n=1}^{\infty}\vert\langle y,f_{2n-1}\rangle \vert^2}\\
&=&\displaystyle{\sum_{n=1}^{\infty}\vert\langle y,y_n\rangle \vert^2}.\\
\end{array}$$
Hence $\{x_n\}_{n\geqslant 1}$ is a $K$-frame for $H_1$ and $\{y_n\}_{n\geqslant 1}$ is a $L$-frame for $H_2$.\ Let's show, now, that $\{x_n\oplus y_n\}_{n\geqslant1}$ is a $K\oplus L$-minimal frame for $H_1\oplus H_2$. Let $\{\delta_n\}_{n\geqslant 1}$ be the standard orthonormal basis for $\ell^2$.\\ \\
We have for all $n\geqslant 1$, $\theta_1(e_n)=\left\lbrace
\begin{array}{rcl}
&0& \text{ if } n \text{ is odd,}\\
&\delta_n& \text{ if } n \text{ is even.}
\end{array}
\right.$\\ \\
That means that $\overline{R(\theta_1)}=\overline{span}\{\delta_{2n},\; n\geqslant 1\}$.\\ \\
We have for all $n\geqslant 1$, $\theta_2(f_n)=\left\lbrace
\begin{array}{rcl}
&0& \text{ if } n \text{ is even,}\\
&\delta_n& \text{ if } n \text{ is odd.}
\end{array}
\right.$\\ \\
That means that $\overline{R(\theta_2)}=\overline{span}\{\delta_{2n-1},\; n\geqslant 1\}$. 
Hence $\overline{R(\theta_1)}=R(\theta_2)^\perp$. Proposition \ref{Prop2.14} implies, then, that $\{x_n\oplus y_n\}_{n\geqslant1}$ is a $K\oplus L$-minimal frame for $H_1\oplus H_2$.
\\
\end{example}

\subsection{$K\oplus L$-orthonormal bases}
\medskip
\noindent

In what follows, we investigate   $K\oplus L$-orthonormal bases for super Hilbert spaces.\\
\begin{proposition}
Let $\{x_n\}_{n\geqslant 1}$ and $\{y_n\}_{n\geqslant 1}$ be two orthonormal systems in $H_1$ and $H_2$ respectively. Then $\{x_n\oplus y_n\}_{n\geqslant 1}$ is never an orthonormal system for the super Hilbert space $H_1\oplus H_2$.
\end{proposition}
\begin{proof}
For any $n\geqslant 1$,  $\langle x_n\oplus y_n,x_n\oplus y_n\rangle =\langle x_n,x_n\rangle+\langle y_n,y_n\rangle=2\neq 1.$\\
\end{proof}

The following result is a necessary condition on $\{x_n\}_{n\geqslant 1}$ and $\{y_n\}_{n\geqslant 1}$ for $\{x_n\oplus y_n\}_{n \geqslant1}$ to be a $K\oplus L$-orthonormal basis.
\begin{proposition}\label{Prop2.23}
Let $K\in B(H_1)$ and $L\in B(H_2)$. Let $\{x_n\oplus y_n\}_{n\geqslant 1}\subset H_1\oplus H_2$.\\
\begin{enumerate}
\item[i. ] If $K=0$ and $L\neq 0$. Then:\\
$\{x_n\oplus y_n\}_{n\geqslant 1}$ is $K\oplus L$-orthonormal basis $\Longleftrightarrow$ $\left\lbrace
\begin{array}{rcl}
&1) &\forall n\geqslant 1, x_n=0.\\
&2) &\{y_n\}_{n\geqslant 1} \text{ is a }\text{$L$-orthonormal}\\
&&\text{ basis for } H_2.
\end{array}
\right.$
\item[ii. ] If $K\neq0$ and $L= 0$. Then:\\
$\{x_n\oplus y_n\}_{n\geqslant 1}$ is $K\oplus L$-orthonormal basis $\Longleftrightarrow$ $\left\lbrace
\begin{array}{rcl}
&1)&\forall n\geqslant 1, y_n=0.\\
&2)&\{x_n\}_{n\geqslant 1} \text{ is a }\text{$K$-orthonormal}\\
&& \text{basis for } H_1.
\end{array}
\right.$\\
\item[iii. ] If $K\neq0$ and $L\neq 0$. If $\{x_n\oplus y_n\}_{n\geqslant 1}$ is a $K\oplus L$-orthonormal basis, then:\\
\begin{enumerate}
\item $\{K^*(x_n)\oplus L^*(y_n)\}_{n\geqslant 1}$ is the unique $K\oplus L$-dual frame to $\{x_n\oplus y_n\}_{n\geqslant 1}$.
\item $\{x_n\}_{n\geqslant 1}$ is  a non-minimal $K$-frame for $H_1$ whose $\{K^*(x_n)\}_{n\geqslant 1}$ is a $K$-dual frame.
\item $\{y_n\}_{n\geqslant 1}$ is a non-minimal $L$-frame for $H_2$ whose $\{L^*(y_n)\}_{n\geqslant 1}$ is a $L$-dual frame.\\
\end{enumerate}
\end{enumerate}
\end{proposition}
\begin{proof}\hspace{0.2cm}
\begin{enumerate}
\item[i. ] Assume that  $K=0$ and $L\neq 0$. And 
assume that $\{x_n\oplus y_n\}_{n\geqslant 1}$ is a \\$K\oplus L$-orthonormal basis. Then for all $x\oplus y\in H_1\oplus H_2$, $$\|L^*(y)\|^2=\|K^*(x)\oplus L^*(y)\|^2=\displaystyle{\sum_{n=1}^{\infty}\vert \langle x\oplus y,x_n\oplus y_n\rangle\vert^2}.$$ Taking $y=0$, we obtain $0=\displaystyle{\sum_{n=1}^{\infty}\vert\langle x,x_n\rangle\vert^2}$. Hence $\langle x,x_n\rangle =0$ for\\ any $x$ in $H_1$. Hence for all $n\geqslant 1$, $x_n=0$. Then for all $y\in H_2$, $\|L^*(y)\|^2\\=\displaystyle{\sum_{n=1}^{\infty}\vert\langle y,y_n\rangle \vert^2}$ and for all $n,m\geqslant 1$, $\langle y_n,y_m\rangle =\langle x_n\oplus y_n,x_m\oplus y_m\rangle=\delta_{n,m}$. \\Hence $\{y_n\}_{n\geqslant 1}$ is a $L$-orthonormal basis for $H_2$.
The converse is clear.\\
\item[ii. ] By the same way followed in i.\\
\item[iii. ] Assume that $K\neq 0$ and $L\neq 0$.\\
By proposition \ref{Prop1.18}, the unique dual to $\{x_n\oplus y_n\}_{n\geqslant}$ is $\{K^*(x_n)\oplus L^*(y_n)\}_{n\geqslant 1}$.\\
By corollary \ref{Cor2.5} and corollary \ref{Cor2.17}, $\{x_n\}_{n\geqslant 1}$ and $\{y_n\}_{n\geqslant 1}$ are non-minimal $K$-frame for $H_1$ and $L$-frame for $H_2$ respectively. By (a) and proposition \ref{Prop2.20}, we have that $\{K^*(x_n)\}_{n\geqslant 1}$ is a $K$-dual frame to $\{x_n\}_{n\geqslant 1}$ and $\{L^*(y_n)\}_{n 1}$ is a $L$-dual frame to $\{y_n\}_{n\geqslant 1}$.

\end{enumerate}
\end{proof}
By the above proposition, we deduce that if $\{x_n\oplus y_n\}_{n\geqslant 1}$ is a $K\oplus L$-orthonormal basis for $H_1\oplus H_2$, then $\{K^*(x_n)\oplus L^*(y_n)\}_{n\geqslant 1}$ is a $K^*\oplus L^*$-frame for $H_1\oplus H_2$. One can wonder under which conditions $\{K^*(x_n)\oplus L^*(y_n)\}_{n\geqslant 1}$ is a $K^*\oplus L^*$-orthonormal basis.\\ \\
we will need the following lemmas.
\begin{lemma}
Let $K\in B(H_1)$ and $L\in B(H_2)$. Then the following statements are equivalent: 
\begin{enumerate}
\item[i. ] $K\oplus L$ is an isometry in  $H_1\oplus H_2$.
\item[ii. ] $K$ and $L$ are both isometries.\\ \\
\end{enumerate}
\end{lemma}
\begin{proof} \hspace{0.2cm}
\begin{enumerate}
\item[i $\Longrightarrow$ ii. ] Assume that $K\oplus L$ is an isometry in  $H_1\oplus H_2$. Then for all $x\oplus y\in H_1\oplus H_2$, we have: $$\begin{array}{rcl}
x\oplus y&=&(K\oplus L)^*(K\oplus L)(x\oplus y)\\
&=& (K^*\oplus L^*)(K(x)\oplus L(y))\\
&=& K^*K(x)\oplus L^*L(y).
\end{array}$$
Hence For all $x\in H_1$, $x=K^*K(x)$ and for all $y\in H_2$, $y=L^*L(y)$. That means that $K^*K=I_{H_1}$ and $L^*L=I_{H_2}$. \\
Hence $K$ and $L$ are both isometries.
\item[ii $\Longrightarrow $ i.] Assume that $K$ and $L$ are both isometries. Then $K^*K=I_{H_1}$ and $L^*L=I_{H_2}$. Hence for all $x\oplus y\in H_1\oplus H_2$, $$\begin{array}{rcl}
(K\oplus L)^*(K\oplus L)(x\oplus y)&=& (K^*\oplus L^*)(K(x)\oplus L(y))\\
&=& K^*K(x)\oplus L^*L(y)\\
&=& x\oplus y.
\end{array}$$
Then $(K\oplus L)^*(K\oplus L)=I_{H_1\oplus H_2}$. That means that $K\oplus L$ is an isometry in $H_1\oplus H_2$.\\
\end{enumerate}
\end{proof}
\begin{lemma}\label{Lem2.25}
Let $K\in B(H_1)$ and $L\in B(H_2)$. Then the following statements are equivalent: 
\begin{enumerate}
\item[i. ] $K\oplus L$ is a co-isometry in  $H_1\oplus H_2$.
\item[ii. ] $K$ and $L$ are both co-isometries.\\
\end{enumerate}
\end{lemma}
\begin{proof}
It is a direct result of the above lamma and the fact that $(K\oplus L)^*=K^*\oplus L^*$.\\

\end{proof}
\begin{proposition}\label{Prop2.26}
Let $K\in B(H_1)$ and $L\in B(H_2)$. \\Let $\{x_n\oplus y_n\}_{n\geqslant 1}$ be a $K\oplus L$-orthonormal basis for the super Hilbert space $H_1\oplus H_2$. Then the following statements are equivalent: 
\begin{enumerate}
\item[i. ] $\{K^*(x_n)\oplus L^*(y_n)\}_{n\geqslant 1}$ is a $K^*\oplus L^*$-orthonormal basis.
\item[ii. ] $K$ and $L$ are both co-isometries.\\
\end{enumerate}
\end{proposition}
	\begin{proof}
	$$\begin{array}{rcl}
	&&\{K^*(x_n)\oplus L^*(y_n)\}_{n\geqslant 1} \text{ is a $K^*\oplus L^*$-orthonormal basis. }\\
	&\Longleftrightarrow& \{K^*\oplus L^*(x_n\oplus y_n)\}_{n\geqslant 1} \text{ is a $K^*\oplus L^*$-orthonormal basis.}\\
	&\Longleftrightarrow& K\oplus L \text{ is a co-isometry (Proposition \ref{Prop1.19})}\\
	&\Longleftrightarrow& K \text{ and } L \text{ are both co-isometries (Lemma \ref{Lem2.25}).}
	\end{array}$$
	
	\end{proof}
The following corollary gives some necessary conditions for the $K\oplus L$-dual frame to a $K\oplus L$-orthonormal basis to be a $K^*\oplus L^*$-orthonormal basis. 
\begin{corollary}
Let $K\in B(H_1)$ and $L\in B(H_2)$. \\Let $\{x_n\oplus y_n\}_{n\geqslant 1}$ be a $K\oplus L$-orthonormal basis for the super Hilbert space $H_1\oplus H_2$. If $\{K^*(x_n)\oplus L^*(y_n)\}_{n\geqslant 1}$ is a $K^*\oplus L^*$-orthonormal basis, then the pair $(\{x_n\}_{n\geqslant 1},\{y_n\}_{n\geqslant 1})$ is a complete strongly disjoint pair of Parseval frames. i.e. $\{x_n\}_{n\geqslant 1}$ and $\{y_n\}_{n\geqslant 1}$ are Parseval frames for $H_1$ and $H_2$ respectively, and $\{x_n \oplus y_n\}_{n\geqslant 1}$ is an orthonormal basis for $H_1\oplus H_2$. In particular, $ R(\theta_1)$ and $R(\theta_2)$ are closed subspaces of $\ell^2$ and  $R(\theta_1)=R(\theta_2)^\perp$ in $\ell^2$,\\
where $\theta_1$ and $\theta_2$ are the frame transforms for $\{x_n\}_{n\geqslant 1}$ and $\{y_n\}_{n\geqslant 1}$ respectively.\\
\end{corollary}	
\begin{proof}
Assume that $\{K^*(x_n)\oplus L^*(y_n)\}_{n\geqslant 1}$ is a $K^*\oplus L^*$-orthonormal basis, then by proposition \ref{Prop2.26}, $K\oplus L$ is a co-isometry. Then for all $x\oplus y\in H_1\oplus H_2$, we have $
\|K^*\oplus L^*(x\oplus y)\|=\|x\oplus y\|.$ Since $\{x_n\oplus y_n\}_{n\geqslant 1}$ is a $K\oplus L$-orthonormal basis, then it is an orthonormal system and a Parseval $K\oplus L$-frame. Thus for all $x\oplus y$ in $ H_1\oplus H_2$, $\|x\oplus y\|^2=\|K^*\oplus L^*(x\oplus y)\|^2 =\displaystyle{\sum_{n=1}^{+\infty}\vert\langle x\oplus y, x_n\oplus y_n\rangle \vert^2}.$ That means that $\{x_n\oplus y_n\}_{n\geqslant 1}$ is a Parseval frame which is also an orthonormal system. Hence $\{x_n\oplus y_n\}_{n\geqslant 1}$ is an orthonormal basis for $H_1\oplus H_2$. Thus $\{x_n\}_{n\geqslant 1}$ and $\{y_n\}_{n\geqslant 1}$ are  Parseval frames for $H_1$ and $H_2$ respectively due to the  proposition \ref{Prop2.1}. Hence $ R(\theta_1)$ and $R(\theta_2)$ are closed subspaces of $\ell^2$. The fact that  $R(\theta_1)=R(\theta_2)^\perp$ in $\ell^2$ follows immediately from  theorem 2.9' in  \cite{8}.

\end{proof}
	\medskip
	\section*{Declarations}
	
	\medskip
	
	\noindent \textbf{Availablity of data and materials}\newline
	\noindent Not applicable.
	
	\medskip

	\noindent \textbf{Competing  interest}\newline
	\noindent The author declares that he has no competing interests.

	\medskip
	
	\noindent \textbf{Fundings}\newline
	\noindent  Author declares that there is no funding available for this article.

\end{document}